\documentclass[11pt]{amsart}
\usepackage{amsmath,amsfonts,amssymb,amsthm,amscd}
\usepackage{graphicx}
\usepackage{mathrsfs}
\usepackage[matrix,arrow,curve]{xy}
\begin{document}

\theoremstyle{plain}
\newtheorem{theorem}{Theorem}
\newtheorem{lemma}{Lemma}
\newtheorem{claim}{Claim}
\newtheorem{proposition}{Proposition}
\newtheorem{corollary}{Corollary}
\newtheorem*{tm}{Theorem}
\newtheorem*{lm}{Lemma}
\newtheorem*{clm}{Claim}
\newtheorem*{prp}{Proposition}
\newtheorem*{cor}{Corollary}

\theoremstyle{definition}
\newtheorem{dfn}{Definition}
\newtheorem{problem}{Problem}
\newtheorem{conj}{Conjecture}
\newtheorem*{df}{Definition}
\newtheorem*{nt}{Notation}
\newtheorem*{obs}{Observation}
\newtheorem*{exm}{Example}
\newtheorem*{exms}{Examples}
\newtheorem*{prb}{Problem}
\newtheorem*{cnj}{Conjecture}

\theoremstyle{remark}
\newtheorem{remark}{Remark}
\newtheorem*{pf}{Proof}
\newtheorem*{rmk}{Remark}
\newtheorem{question}{Question}
\newtheorem*{q}{Question}
\newtheorem*{qs}{Questions}
\newtheorem*{ex}{Example}
\newtheorem*{exs}{Examples}
\newtheorem*{ack}{Acknowledgements}
\newtheorem*{fact}{Fact}
\newtheorem*{facts}{Facts}

\title{On two types of ultrafilter extensions 
of binary relations}
\author{Denis~I.~Saveliev}%
\date{}

\thanks{
\noindent
This work was partially supported by grant 
17-01-00705 of Russian Foundation for Basic Research 
and carried out at Institute for Information 
Transmission Problems of the Russian Academy of Sciences.
}



\maketitle

\newcommand{\ob}{[}
\newcommand{\cb}{]}

\newcommand{\dom}{ {\mathop{\mathrm {dom\,}}\nolimits} }
\newcommand{\ran}{ {\mathop{\mathrm{ran\,}}\nolimits} }

\newcommand{\cf}{ {\mathop{\mathrm {cf\,}}\nolimits} }
\newcommand{\inter}{ {\mathop{\mathrm {int\,}}\nolimits} }
\newcommand{\cl}{ {\mathop{\mathrm {cl\,}}\nolimits} }
\newcommand{\cll}{ {\mathop{\mathrm {cl^{-}\,}}\nolimits} }
\newcommand{\clu}{ {\mathop{\mathrm {cl^{+}\,}}\nolimits} }
\newcommand{\lcl}{ {\mathop{\mathrm {lcl\,}}\nolimits} }
\newcommand{\rcl}{ {\mathop{\mathrm {rcl\,}}\nolimits} }
\newcommand{\cof}{ {\mathop{\mathrm {cof\,}}\nolimits} }
\newcommand{\add}{ {\mathop{\mathrm {add\,}}\nolimits} }
\newcommand{\sat}{ {\mathop{\mathrm {sat\,}}\nolimits} }
\newcommand{\tc}{ {\mathop{\mathrm {tc\,}}\nolimits} }
\newcommand{\unif}{ {\mathop{\mathrm {unif\,}}\nolimits} }
\newcommand{\uhr}{\!\upharpoonright\!}
\newcommand{\lra}{ {\leftrightarrow} }
\newcommand{\ot}{ {\mathop{\mathrm {ot\,}}\nolimits} }
\newcommand{\pr}{ {\mathop{\mathrm {pr}}\nolimits} }
\newcommand{\cnc}{ {^\frown} }
\newcommand{\image}{\/``\,}
\newcommand{\scc}{\beta\!\!\!\!\beta}
\newcommand{\filt}{\varepsilon\!\!\!\!\!\;\varepsilon}
\newcommand{\pws}{P\!\!\!\!\!P}
\newcommand{\ol}{\overline}
\newcommand{\wh}{\widehat}
\newcommand{\wt}{\widetilde}

\begin{abstract}
There exist two distinct types of ultrafilter extensions 
of binary relations, one discovered in universal algebra 
and modal logic, and another, in model theory and algebra 
of ultrafilters. We show that the extension of the latter 
type is properly included in the extension of the former 
type, and describe their interaction with the relation 
algebra operations. Then we provide topological 
characterizations of both extensions and show that the 
larger extension continuously maps the space of ultrafilters 
into the space of filters endowed with the Vietoris topology. 
\end{abstract}

\section{Introduction}

There exist two known ways to extend a~binary relation~$R$ 
on a~set~$X$ to a~binary relation on the set $\scc X$ 
of ultrafilters over~$X$ (where $X$~is identified 
with the subset of~$\scc X$ consisting of principal 
ultrafilters), both canonical in a~certain sense. We 
denote these two extensions of $R$ by $\wt R$ and~$R^*$.
They are defined by letting, for all $u,v\in\scc X$,
\begin{align*}
\wt R(u,v) 
&\;\;\lra\;\;
\{x:\{y:R(x,y)\}\in v\}\in u,
\\
R^*(u,v) 
&\;\;\lra\;\;
(\forall S\in v)\;
\{x:(\exists y\in S)\:R(x,y)\}\in u.
\end{align*}

The extension~$R^*$ comes from universal algebra 
\cite{Jonsson Tarski}, later \cite{Goldblatt}, 
\cite{Goranko}, and modal logic \cite{Lemmon} 
where it is used e.g.~to characterize modal 
definability~\cite{Goldblatt Thomason} (see 
\cite{Lemmon Scott}, \cite{van Benthem},~\cite{Venema}). 
The extension~$\wt{R}$ is recently discovered in 
model theory \cite{Saveliev}, \cite{Saveliev expanded} 
and has as precursors multiplication of ultrafilters 
arising in iterated ultraproducts \cite{Kochen}, 
\cite{Frayne Morel Scott} (see~\cite{Chang Keisler})
and especially so-called algebra of ultrafilter essentially 
dealing with ultrafilter extensions of semigroups and known 
by its numerous applications (see~\cite{Hindman Strauss}).

In fact, in \cite{Jonsson Tarski}, \cite{Goldblatt}, 
\cite{Goranko} the ultrafilter extension~$R^*$ is defined 
for $n$-ary relations~$R$, and also in \cite{Saveliev}, 
\cite{Saveliev expanded} the ultrafilter extension~$\wt{R}$ 
is defined for $n$-ary~$R$, with any finite~$n$. 
Moreover, \cite{Goranko} and \cite{Saveliev}, 
\cite{Saveliev expanded} define also (one type of) 
ultrafilter extension of $n$-ary maps (which generalizes 
continuous extension of unary maps as well as the mentioned 
constructions of multiplication of ultrafilters and of 
ultrafilter extension of semigroups). Thus in fact  
\cite{Goranko}, \cite{Saveliev}, \cite{Saveliev expanded} 
study two types of ultrafilter extensions of arbitrary 
first-order models.

In this article, however, we deal only with binary
relations, and with both types of their extensions.
The extension~$\wt{R}$ of binary relations~$R$ of 
special form, namely, linear orders, are studied 
in~\cite{Saveliev(orders)}.
For basic definitions and properties of both types of 
ultrafilter extensions of first-order models, as well 
as for some historical information, we refer the reader 
to the introductory part of~\cite{Poliakov Saveliev 2018}. 
Some further information of ultrafilter extensions of 
first-order models can be found in \cite{Saveliev Shelah} 
and \cite{Poliakov Saveliev 2018}.

This article goes as follows. 
In Section~2, we study the relationships of the two types 
of ultrafilter extension of binary relations, and describe 
their interplay with operations of relation algebra. We 
also characterize the~${}^*$-extension via projections.
In Section~3, we provide topological characterizations of 
both extensions. The~${}^*$-extension is characterized 
via the closure operation in the square $\scc X\times\scc X$, 
while the~$\wt{\;}$-extension via certain operations of left 
and right closure, and also via continuous extensions of 
homomorphisms. We also consider there generalizations of 
the left and right closure to arbitrary topological spaces. 
In Section~4, we show that the~${}^*$-extension of relations 
on $X$ can be considered as continuous extension of certain 
associated maps into the space~$\filt X$ of filters over~$X$ 
endowed with the Vietoris topology 
(by identifying the filters with closed sets in~$\scc X$).


\section{Relationships between $R^*$ and $\wt R$ and relation algebra\/}

In this section we study relationships between two extensions,
${}^*$ and~$\,\wt{}\;\;$, of binary relations, and their interplay 
with operations of relation algebra.

\subsection*{Relationships between $R^*$ and $\wt R$\/}

The following lemma contains a~half of information required for 
Theorem~\ref{2.2}. Note that in it, (0)~describes a~given~$R$ as 
a~relation on~$\scc X$, while (i) and~(v)~describe $R^*$ and~$\wt R$,
respectively.

\begin{lemma}\label{2.1}
Let $R\subseteq X\!\times X$ and $u,v\in\scc X$. 
The implications between the following formulas 
\begin{align*}
\phantom{iiiii}\text{{\rm (0)}}\quad&
(\exists x,y\in X)\;
(\{x\}\in u\wedge\{y\}\in v\wedge R(x,y)),
\\
\phantom{iiiii}\text{{\rm (i)}}\quad&
(\forall S\in v)\;
\{x:(\exists y\in S)\:R(x,y)\}\in u,
\\
\phantom{iiii}\text{{\rm (ii)}}\quad&
(\forall S\in u)\;
\{y:(\exists x\in S)\:R(x,y)\}\in v,
\\
\phantom{iii}\text{{\rm (iii)}}\quad&
(\exists S\in v)\;
\{x:(\forall y\in S)\:R(x,y)\}\in u,
\\
\phantom{iii}\text{{\rm (iv)}}\quad&
(\exists S\in u)\;
\{y:(\forall x\in S)\:R(x,y)\}\in v,
\\
\phantom{iiii}\text{{\rm (v)}}\quad&
\quad\;\;\;\{x:\{y:R(x,y)\}\in v\}\in u,
\\
\phantom{iii}\text{{\rm (vi)}}\quad&
\quad\;\;\;\{y:\{x:R(x,y)\}\in u\}\in v
\end{align*}
are as follows:
\begin{align*}
\begin{CD}
@.@.\text{{\rm (vi)}}@>>>\text{{\rm (i)}}@=\text{{\rm (ii)}}
\\
@.@.@AAA@AAA@.
\\
\text{{\rm (0)}}@>>>\text{{\rm (iii)}}@=\text{{\rm (iv)}}@>>>\text{{\rm (v)}}@.
\end{CD}
\end{align*}
\rule{0em}{0.5em}
\end{lemma}

\begin{proof}
(0)$\to$(iii)\;
Easy.

(i)$\to$(ii)\;
Let $u,v\in\scc X$ be such that (i)~holds. 
Toward a~contradiction, assume that (ii)~fails:
$$
(\exists S\in u)\;
\{y:(\exists x\in S)\:R(x,y)\}\notin v.
$$
Since $v$~is ultra, 
this is equivalent to 
$$
(\exists S\in u)\;
\{y:(\forall x\in S)\:\neg\,R(x,y)\}\in v.
$$
Fix an~$S$ witnessing this. Denote the set 
$\{y:(\forall x\in S)\:\neg\,R(x,y)\}$ by~$B$. 
We have $B\in v$. Therefore, by~(i),
$$
\{x:(\exists y\in B)\:R(x,y)\}\in u. 
$$
Denote the set $\{x:(\exists y\in B)\:R(x,y)\}$ by~$A$. 
We have $A\in u$. Hence $A\cap S\in u$, so pick some 
$x_0\in A\cap S$. On the one hand, since $x_0\in A$,
there exists $y_0\in B$ such that $R(x_0,y_0)$. 
On the other hand, we have $\neg R(x_0,y_0)$ since 
$x_0\in S$ and $y_0\in B$. A~contradiction.

(ii)$\to$(i)\;
This follows from (i)$\to$(ii) since 
(i)~turns into~(ii) and conversely by interchanging 
$x,u$ with $y,v$ and taking $R^{-1}$ rather than~$R$.

(iii)$\,\lra\,$(iv)\;
This follows from (i)$\,\lra\,$(ii) since 
(iii) and (iv)~are the negations of (i) and~(ii) 
applied to $\neg\,R$ rather than~$R$.

(v)$\to$(i)\;
Let $u,v\in\scc X$ be such that (v)~holds. 
Denote the set $\{x:\{y:R(x,y)\}\in v\}$ by~$A$ 
and, for any given~$x$, the set $\{y:R(x,y)\}$ by~$A_x$. 
We have $A\in u$ and $A_x\in v$ whenever $x\in A$.

Toward a~contradiction, assume that (i)~fails: 
$$
(\exists S\in v)\;
\{x:(\exists y\in S)\:R(x,y)\}\notin u.
$$
Since $u$~is ultra, this is equivalent to 
$$
(\exists S\in v)\;
\{x:(\forall y\in S)\:\neg\,R(x,y)\}\in u.
$$
Fix an~$S$ witnessing this and denote the set 
$\{x:(\forall y\in S)\,\neg\,R(x,y)\}$ by~$B$. 
We have $B\in u$.

Since $A,B\in u$, we get $A\cap B\in u$, so pick some 
$x_0\in A\cap B$. It follows from $x_0\in A$ that 
$A_{x_0}\in v$, whence $A_{x_0}\cap S\in v$, 
so pick some $y_0\in A_{x_0}\cap S$.

On the one hand, we have $R(x_0,y_0)$ 
since $x_0\in A$ and $y_0\in A_{x_0}$. 
On the other hand, we have $\neg\,R(x_0,y_0)$ 
since $x_0\in B$ and $y_0\in S$. A~contradiction.

(vi)$\to$(ii)\;
This follows from (v)$\to$(i) since
(i)~turns into~(ii) and (v)~into~(vi) by interchanging 
$x,u$ with $y,v$ and taking $R^{-1}$ rather than~$R$.

(iii)$\to$(v)\;
This follows from (i)$\,\lra\,$(v) since 
(iii) and (v)~are the negations of (i) and (v)~itself 
applied to $\neg\,R$ rather than~$R$. 
\end{proof}


\subsection*{The extensions versus relation algebra\/}

In the sequel, we denote by~$-$ the complement operation; 
e.g.$-((-R)^*)$ written explicitly is 
$(\scc X\!\times\scc X)\!\setminus\!((X\!\times X)\!\setminus\!R)^*$.

\begin{theorem}\label{2.2}
The inclusions (depicted by arrows) 
between any given $R\subseteq X\!\times X$ and its extensions 
$R^*$, $\wt R$, $-((-R)^*)$, $(\wt{R^{-1}})^{-1}$ are as follows:
\begin{align*}
\begin{CD}
@.(\wt{R^{-1}})^{-1}@>>>R^*@=((R^{-1})^*)^{-1}
\\
@.@AAA@AAA@.@.
\\
R@>>>-((-R)^*)@>>>\wt R@=-(\wt{-R})
\end{CD}
\end{align*}
\rule{0em}{0.5em}%
In general, these inclusions are non-reversible; moreover, 
each of all possible relationships between the extensions 
is satisfied in some model. 
\end{theorem}

\begin{proof}
By definition, $R^*$~consists of the pairs $(u,v)$ satisfying~(i), 
while $\wt R$~consists of the pairs $(u,v)$ satisfying~(v). 
And, as easy to see, in the same manner 
(ii)~defines $((R^{-1})^*)^{-1}$,
(iii)~defines $-((-R)^*)$, 
(iv)~defines $-((-R^{-1})^*)^{-1}$,
and (vi)~defines $(\wt{R^{-1}})^{-1}$. 
Now the inclusions are immediate from Lemma~\ref{2.1}; let us comment 
only that $-((-R)^*)\subseteq(\wt{R^{-1}})^{-1}$ follows from 
(v)$\to$(i)$\,\lra\,$(ii) since (vi) to~(ii) is the same that 
(v) to~(i).

It remains to provide examples of all possible relationships between 
the four discussed extensions. Below $X$~is an infinite set; $<$ and 
$\le$~denote a~strict and non-strict linear orders resp., $=_Y$~the 
equality relation, and $U_Y$~the universal relation on a~set~$Y$. 
\begin{align*}
\begin{array}{c|cccccccc}
\quad\;\; R\qquad
&\quad\;-((-R)^*)
&\qquad(\wt{R^{-1}})^{-1}
&\;\wt R
&\!\!\!\!\!\!\!\!R^*
\\
\hline
\\
\emptyset\;\;
&\;\;\emptyset
&\;\;\emptyset
&\;\emptyset
&\;\emptyset
\phantom{\underbrace{x}}
\\
U_X
&\quad\;\; U_{\scc X}
&\quad\;\; U_{\scc X}
&\quad\; U_{\scc X}
&\quad\; U_{\scc X}
\phantom{\underbrace{x}}
\\
=_X
&\quad=_X
&\quad=_X
&\quad\;=_X
&\qquad=_{\scc X}
\phantom{\underbrace{|}}
\\
U_X\setminus\!=_X
&\quad U_{\scc X}\setminus\!=_{\scc X}
&\quad U_{\scc X}\setminus\!=_X
&\quad U_{\scc X}\setminus\!=_X
&\quad U_{\scc X}\setminus\!=_X
\phantom{\underbrace{|}}
\\
<\;\;
&\;\;\vartriangleleft
&\quad\;(\wt{>})^{-1}
&\quad\wt{<}
&\quad\;\;\trianglelefteq\setminus\!=_X
\phantom{\underbrace{|}}
\\
\le\;\;
&\quad\;\vartriangleleft\cup=_X
&\quad\;(\wt{\ge})^{-1}
&\quad\wt{\le}
&\quad\;\;\trianglelefteq\qquad
\end{array}
\end{align*}
\rule{0em}{0.5em}%
All these extensions are calculated immediately, 
except for the cases of $<$ and~$\le$. As for the latters, 
note that the relations $\wt{<}$ and $(\wt{>})^{-1}$ 
are incomparable by inclusion, their intersection 
is~$\vartriangleleft$, and their union is 
$\trianglelefteq\setminus\!=_X$; the relationships 
between the extensions of~$\le$ are analogous. The extensions 
of linear orders can be described in terms of so-called 
supports of ultrafilters over linearly ordered sets; we refer 
the reader to~\cite{Saveliev(orders)} for a~necessary information 
about the supports, the extensions $\wt{<}$ and~$\wt{>}$, 
and the relations $\vartriangleleft$ and~$\trianglelefteq$. 
\end{proof}


The facts concerning distributivity of the extensions w.r.t.~the 
operations of relation algebra are listed in the next theorem.

\begin{theorem}\label{2.3}
Let\, $R,S\subseteq X\!\times X$.
The following equalities hold:
\begin{align*}
-\wt R&=\wt{-R},
&
\wt{R\cup S}&=\wt R\cup\wt S,
&
\wt{R\cap S}&=\wt R\cap\wt S.
\\
(R^*)^{-1}&=(R^{-1})^*,
&
(R\cup S)^*&=R^*\cup S^*,
&
(R\circ S)^*&=R^*\circ S^*,
\end{align*}
Moreover, the following inclusions hold:
$$
-\,R^*\subseteq(-R)^*,\qquad
(R\cap S)^*\subseteq R^*\cap S^*,\qquad
\wt R\circ\wt S\subseteq\wt{R\circ S},
$$
neither of them, in general, cannot be replaced by an equality. 
Finally, $\wt R^{-1}$ and $\wt{R^{-1}}$, in general, are 
$\subseteq$-incomparable.
\end{theorem}

\begin{proof}
1. 
The equalities $(R^*)^{-1}=(R^{-1})^*$ 
and $-\wt R=\wt{-R}$ are immediate from
$R^*=((R^{-1})^*)^{-1}$ and $\wt R=-(\wt{-R})$, 
which we have already observed. Also the equality
$\wt{R\cap S}=\wt R\cap\wt S$ is clear as for all
$u,v\in\scc X$,
\begin{align*}
(u,v)\in\wt{R\cap S}
&\;\;\lra\;\;
\bigl\{x:\{y:R(x,y)\wedge S(x,y)\}\in v\bigr\}\in u
\\
&\;\;\lra\;\;
\bigl\{x:\{y:R(x,y)\}\in v\wedge
\{y:S(x,y)\}\in v\bigr\}\in u
\\
&\;\;\lra\;\;
\bigl\{x:\{y:R(x,y)\}\in v\bigr\}\in u
\,\wedge\, 
\bigl\{x:\{y:S(x,y)\}\in v\bigr\}\in u
\end{align*}
(the second equivalence holds since $v$~is a~filter 
while the third since so is~$u$). It follows that 
$\,{}\wt{}\;$~distributes with all Boolean connectives. 
Let us handle two remaining equalities.

Show $(R\cup S)^*=R^*\cup S^*$. We have 
(writting~${\;}^*\,$ as in (ii)~in Lemma~\ref{2.1}) 
\begin{align*}
(u,v)\in(R\cup S)^*
&\;\;\lra\;\;
(\forall A\in u)\;
\{y:(\exists x\in A)\:R(x,y)\,\vee\,S(x,y)\}\in v
\\
&\;\;\lra\;\;
(\forall A\in u)\;
\bigl(\/
\{y:(\exists x\in A)\:R(x,y)\}\in v
\\
&
\quad\qquad\qquad\qquad
\:\vee\:
\{y:(\exists x\in A)\:S(x,y)\}\in v
\/\bigr)
\end{align*}
and 
\begin{align*}
(u,v)\in R^*\cup S^*
\;\;\lra\;\;
&(\forall A\in u)\;
\{y:(\exists x\in A)\:R(x,y)\}\in v
\\
&\;\vee\;
(\forall A\in u)\;
\{y:(\exists x\in A)\:S(x,y)\}\in v,
\end{align*}
whence $(R\cup S)^*\supseteq R^*\cup S^*$ clearly follows.
To prove the converse inclusion, assume that it fails, so 
there exists $A_0\in u$ such that either 
$\{y:(\exists x\in A_0)\,R(x,y)\}\notin v$ or
$\{y:(\exists x\in A_0)\,S(x,y)\}\notin v$; 
assume e.g.~that the first holds. But then for any 
$A\in u$, we have $A_0\cap A\in v$, and therefore,
$\{y:(\exists x\in A_0\cap A)\,S(x,y)\}\in v$, 
and so $\{y:(\exists x\in A)\,S(x,y)\}\in v$
a~fortiori, as required.


Show $(R\circ S)^*=R^*\circ S^*$. The inclusion~$\supseteq$ is clear. 
To handle the converse inclusion, we firstly establish the following 
auxiliary fact: For all $B\subseteq X$,
\begin{align*}
R(-R^{-1}B)\subseteq B
\end{align*}
(where $SA$~means the image of~$A$ under~$S$, 
i.e.~the set $\{y:(\exists x\in A)\,S(x,y)\}$). Indeed, 
$R^{-1}B=\{x:(\exists y\in B)\,R(x,y)\}$, hence,
$-R^{-1}B=\{x:(\forall y\in B)\,\neg\,R(x,y)\}$ whence 
\begin{align*}
R(-R^{-1}B)
&=\{z:(\exists x\in-R^{-1}B)\,R(x,z)\}
\\
&=\{z:(\exists x)\,(\forall y\in B)\,\neg\,R(x,y)\wedge R(x,z)\},
\end{align*}
and therefore,
\begin{align*}
-R(-R^{-1}B)
&=\{z:(\forall x)\,(\exists y\in B)\,R(x,y)\vee\neg\,R(x,z)\}
\\
&=\{z:(\forall x)\,(R(x,z)\to(\exists y\in B)\,R(x,y))\},
\end{align*}
which obviously includes~$B$.

Now pick $u,v\in\scc X$ such that $(u,v)\in(R\circ S)^*$ and 
show $(u,v)\in R^*\circ S^*$. Let 
\begin{align*}
W=\{SA:A\in u\}\cup\{R^{-1}B:B\in v\}.
\end{align*}
Check that $W$~has the finite intersection property. Clearly, 
it suffices to verify that $SA\cap R^{-1}B$ is non-empty for 
any $A\in u$ and $B\in v$. Toward a~contradiction, assume the 
opposite. Then $SA\subseteq-R^{-1}B$ and hence 
\begin{align*}
(R\circ S)A=R(SA)\subseteq R(-R^{-1}B)\subseteq B
\end{align*}
(the latter inclusion was established above), thus showing
that $((R\circ S)A)\cap B$ is empty. However, we have $(R\circ S)A\in v$ 
(by our conditions $A\in u$ and $(u,v)\in(R\circ S)^*$, look at (ii)~in 
Lemma~\ref{2.1}) and $B\in v$, hence $((R\circ S)A)\cap B\in v$; a~contradiction. 
This proves that $W$~has the finite intersection property.

Now extend $W$ to some $w\in\scc X$. We have $(u,w)\in S^*$,
$(v,w)\in (R^{-1})^*$, hence $(w,v)\in R^*$ (since ${\,}^*$ 
and~${\,}^{-1}$ commute, as we have already proved), and 
therefore $(u,v)\in R^*\circ S^*$, as required.


2. 
Let us now prove the claims about the inclusions and non-inclusions.

First, $-\,R^*\subseteq(-R)^*$ follows from
$-((-R)^*)\subseteq R$, which we have already observed.

Moreover, if $R$~is $=_X$ then $-\,R^*$~is 
$(\scc X\!\times\scc X)\setminus\!=_{\scc X}$ while 
$(-R)^*$~is $(\scc X\!\times\scc X)\setminus\!=_{X}$, 
which gives an example of $-\,R^*\ne(-R)^*$.

Next, $(R\cap S)^*\subseteq R^*\cap S^*$ is clear since
\begin{align*}
(u,v)\in(R\cap S)^*
&\;\;\lra\;\;
(\forall A\in u)\;
\{y:(\exists x\in A)\:R(x,y)\wedge S(x,y)\}\in v
\end{align*}
and 
\begin{align*}
(u,v)\in R^*\cap S^*
\;\;\lra\;\;
&(\forall A\in u)\;
\bigl(\/
\{y:(\exists x\in A)\:R(x,y)\}\in v
\\
&\qquad\qquad\qquad\;\wedge\;
\{y:(\exists x\in A)\:S(x,y)\}\in v
\/\bigr)
\\
\;\;\lra\;\;
&(\forall A\in u)\;
\{y:(\exists x\in A)\:R(x,y)
\,\wedge\,
(\exists x\in A)\:S(x,y)\}\in v 
\end{align*}
(the last equivalence holds since $v$~is a~filter).

Moreover, if $R$~is $=_X$ and $S$~is
$(X\!\times X)\setminus\!=_{X}$, then 
$(R\cap S)^*$ is~$\emptyset$ while 
$R^*\cap S^*$ is~$=_{\scc X}\!\!\setminus\!=_X$,
which gives an example of $(R\cap S)^*\ne R^*\cap S^*$.

Show $\wt{R}\circ\wt{S}\subseteq\wt{R\circ S}$.
We have 
\begin{align*}
(u,v)\in\wt{R}\circ\wt{S}
&\;\;\lra\;\;
(\exists w\in\scc X)\;
\bigl(\/
\bigl\{x:\{z:R(x,z)\}\in w\bigr\}\in u
\\
&\qquad\qquad\qquad\qquad\:\wedge\:
\bigl\{z:\{y:S(z,y)\}\in v\bigr\}\in w
\/\bigr)
\end{align*}
and 
\begin{align*}
(u,v)\in\wt{R\circ S}
&\;\;\lra\;\;
\bigl\{x:\{y:
(\exists z)\;
R(x,z)\}\,\wedge\,S(z,y)\}\in v\bigr\}\in u.
\end{align*}
Suppose $(u,v)\in\wt R\circ\wt S$ and pick $w\in\scc X$ 
witnessing this. Let $A$~be the set $\{x:\{z:R(x,z)\}\in w\}$ 
and $B$~the set $\{z:\{y:S(z,y)\}\in v\}$. We have
$A\in u$, $B\in w$, and $\{z:R(x,z)\}\in w$ whenever $x\in A$. 
Hence $\{z:R(x,z)\}\cap B\in w$ whenever $x\in A$, so
\begin{align*}
\bigl\{x:\{z:\{y:
R(x,z)\,\wedge\,S(z,y)
\}\in v\bigr\}\in w\bigr\}\in u,
\end{align*}
whence 
\begin{align*}
\bigl\{x:\{y:(\exists z)\;
R(x,z)\}\,\wedge\,S(z,y)\}\in v\bigr\}\in u
\end{align*}
and thus $(u,v)\in\wt{R\circ S}$ follows.

Moreover, if $R$~is $=_X$ and $S$~is $X\!\times X$, 
then $\wt{R}\circ\wt{S}$ is $X\!\times\scc X$ while 
$\wt{R\circ S}$ is $\scc X\!\times\scc X$, which gives 
an example of $\wt{R}\circ\wt{S}\ne\wt{R\circ S}$.

Finally, if $R$~is a~linear order that is not a~well-order, 
then $\wt{R}^{-1}$ and $\wt{R^{-1}}$ are $\subseteq$-incomparable;
see~\cite{Saveliev(orders)} for details (actually, the fact that 
${}^{-1}$ and~${\,}\wt{\;}{\:}$ do not commute was already pointed out 
in~\cite{Saveliev(orders)}, p.~35). 
\end{proof}

\begin{remark}\label{bigcup}
It can be seen from the proof that in fact the $^*$-extension
distributes w.r.t. arbitrary unions and intersections:
$\bigl(\bigcup_iR_i)^*=\bigcup_iR_i^*$ and 
$\bigl(\bigcap_iR_i)^*\subseteq\bigcup_iR_i^*$.
It follows that it distributes w.r.t.~the transitive 
closure of relations. Besides, it distributes w.r.t.~the 
reflexive closure (since $=_{X}{\!\!\!}^*$~equals~$=_{\scc X}$, 
as pointed above).
As for the $\widetilde{}\;$-extension, the subextension consisting
of $\kappa$-complete ultrafilters distributes w.r.t.~$\kappa$-ary 
Boolean operations. It follows that the subextension consisting of 
$\sigma$-complete ultrafilters also distributes w.r.t.~the transitive 
closure.
\end{remark}


We partly summarize the information about distributivity of the extensions 
w.r.t.~operations of relation algebra in the following table, where 
we mark~$1$ if a~given instance of distributivity holds, and $0$~otherwise.
\begin{align*}
\begin{array}{c|ccccc}
\phantom.\;\;&\,\;\;-&\,\;\cap&\,\;\cup&\;\;\circ&\,\;{}^{-1}
\\
\hline
\rule{0em}{1.5em}
\wt{}\;\;&\;\;1&\;1&\;1&\;0&\;0
\\
{}^*\;\;&\;\;0&\;0&\;1&\;1&\;1
\end{array}
\end{align*}
\rule{0em}{0.5em}
We see that the two extensions have, in a~sense, an opposite character 
w.r.t.~these operations: $\;\wt{}\;$~distributes with Boolean operations
but not with $\circ$ and~${}^{-1}$, while, conversely, ${}^*$~distributes 
with the ``inverse-semigroup" fragment of relation algebra but not its 
``Boolean" fragment.

We note also that both extensions distribute w.r.t.~homomorphisms, 
in the sense that any homomorphism of $(X,R)$ into $(Y,S)$ extends
to a~homomorphism of $(\scc X,\wt R)$ into $(\scc Y,\wt S)$ as well 
as to a~homomorphism of $(\scc X,R^*)$ into $(\scc Y,S^*)$. Actually, 
this is a~partial case of general facts about arbitrary first-order
models; see \cite{Saveliev} and~\cite{Goranko} for the ${\;}\wt{}\;\;$
and ${\,}^*\,$~extensions, respectively.


\subsection*{Characterizing of $R^*$ via projections}

\begin{lemma}\label{2.4}
For any $R\subseteq X\times X$ and $u,v\in\scc X$,
$$
R^*(u,v)
\;\;\lra\;\;
(\forall A\in u)\,
(\forall B\in v)\;
R\cap(A\times B)\ne\emptyset.
$$
\end{lemma}

\begin{proof}
In the right direction, if $R^*(u,v)$ then by Lemma~\ref{2.1}, 
for all $A\in u$ we have $RA\in v$, and hence for any $B\in v$
we have $B\cap RA\in v$, whence
\begin{align*}
R\cap(A\times B)=
R\cap(A\times(B\cap RA))\ne
\emptyset.
\end{align*}
In the converse direction, pick $A\in u$. Then for all $B\in v$ we have 
$RA\cap B\ne\emptyset$, hence $RA\in v$ since $v$~is an ultrafilter.
\end{proof}

\begin{theorem}\label{2.5}
Let $R\subseteq X\times X$ and $u,v\in\scc X$. Then $R^*(u,v)$ 
iff there exists $w\in\scc(X\times X)$ such that $R\in w$ and 
$$
\{\pr_1(S):S\in w\}=u
\;\;\text{ and }\;\;
\{\pr_2(S):S\in w\}=v.
$$
\end{theorem}

\begin{proof}
``If". 
Let $w\in\scc(X\times X)$ be such that $R\in w$ and $(\pr_1w,\pr_2w)=(u,v)$.
Pick any $A\in u$, $B\in v$, and show $R\cap(A\times B)\ne\emptyset$; 
by Lemma~\ref{2.4}, the latter is equivalent to $R^*(u,v)$.

Since $A\in u$ and $\pr_1w=\{\pr(S):S\in w\}=u$, it follows that there is
$S_1\in w$ such that $\pr_1(S_1)=A$. Similarly, there is $S_2\in w$ such 
that $\pr_2(S_2)=B$. Since $R\in w$, also $R\cap S_1\cap S_2\in w$. But 
$S_1\cap S_2\subseteq\pr_1(S_1)\times\pr_2(S_2)=A\times B$, so 
$R\cap(A\times B)\in w$, and thus $R\cap(A\times B)\ne\emptyset$ a~fortiori.

``Only if". 
Suppose $R\cap(A\times B)\ne\emptyset$ for any $A\in u$, $B\in v$, 
which means $R^*(u,v)$ by Lemma~\ref{2.4}, and show that there is
$w\in\scc(X\times X)$ as required. Since 
$
\bigcap_{i<n}(A_i\times B_i)=
\bigl(\bigcap_{i<n}A_i\bigr)
\times
\bigl(\bigcap_{i<n}B_i\bigr)
$
if $n<\omega$, the family 
$$
D=\{R\cap(A\times B):{A\in u}\wedge B\in v\}
$$ 
is centered. Pick any $w\in\scc(X\times X)$ that extends~$D$ and show 
that $w$~is as required.

Clearly, if $A\in u$, $B\in v$, then $A\times B\in w$. Hence, if $S\in w$ 
then $S\cap(A\times B)\in w$, whence $A\cap\pr_1(S)\ne\emptyset$ and 
$B\cap\pr_2(S)\ne\emptyset$. Thus for any $A\in u$ we have 
$A\cap\pr_1(S)\ne\emptyset$ whence it follows that $\pr_1(S)\in u$ as 
$u$~is ultra; and similarly $\pr_2(S)\in v$. This shows 
$(\pr_1w,\pr_2w)=(u,v)$, completing the proof.
\end{proof}


\section{Topological characterizations of $R^*$ and $\wt R$\/}

\subsection*{Characterization of $R^*$ via closure}

We consider $\scc X$ endowed with the standard topology generated 
by sets $\{u\in\scc X:S\in u\}$ for all $S\subseteq X$, and 
$\scc X\!\times\scc X$ endowed with the standard product topology.
The closure and the interior operations of $\scc X\!\times\scc X$
are denoted by $\cl{}$ and $\inter{}$~resp.

\begin{theorem}\label{3.1}
If $R\subseteq X\!\times X$ then $R^*=\cl R$ and $-((-R)^*)=\inter(\cl R)$.
\end{theorem}

\begin{proof}
First we observe that $R$~is dense in~$R^*$. Indeed, let $R^*(u,v)$,
i.e.~$\{x:(\exists y\in S)\,R(x,y)\}\in u$ for each $S\in v$, and 
let $U$~be a~neighborhood of $(u,v)$. W.l.g.~$U$~is given by some 
$A\in u$ and $B\in v$, i.e.~$U=\{(u',v'):A\in u'\text{ and }B\in v'\}.$
Taking $B$ as~$S$, we get $\{x:(\exists y\in B)\,R(x,y)\}\in u$ and so
$A\cap\{x:(\exists y\in B)\,R(x,y)\}\in u$, thus finding some $x\in A$ 
and $y\in B$ such that $R(x,y)$, as required.

We complete the proof of the first formula by verifying that $R^*$~is closed. 
Let $u'$ and~$v'$ be such that $(u',v')\in\cl R^*$, so for all $A\in u'$ and 
$B\in v'$ there exist $u$ and~$v$ such that $A\in u$, $B\in v$, and $R^*(u,v)$. 
To show $(u',v')\in R^*$, pick any $B\in v'$ and check that 
$\{x:(\exists y\in B)\,R(x,y)\}\in u'$. Assume the converse: 
$$
\{x:(\exists y\in B)\,R(x,y)\}\notin u'.
$$ 
Since $u'$~is ultra, this is equivalent to
$$
\{x:(\forall y\in B)\,\neg\,R(x,y)\}\in u'.
$$
Hence there exist $u$ and~$v$ such that 
$\{x:(\forall y\in B)\,\neg\,R(x,y)\}\in u$, $B\in v$, and $R^*(u,v)$.
However, $B\in v$ and $R^*(u,v)$ imply 
$\{x:(\exists y\in B)\,R(x,y)\}\in u$. 
A~contradiction.

The second formula follows from the first one. Indeed, as $-((-R)^*)$ means 
then $(\scc X\!\times\scc X)\setminus\cl((X\!\times X)\setminus R)$ while 
$\inter(\cl R)$ is $(\scc X\!\times\scc X)\setminus\cl((\scc X\!\times\scc X)
\setminus\cl R)$, it suffices to verify
$$
\cl((X\!\times X)\setminus R)=
\cl((\scc X\!\times\scc X)\setminus\cl R).
$$
But the inclusion~$\supseteq$ is a~general fact (we have 
$\cl(A\setminus B)\supseteq\cl(\cl A\setminus\cl B)$ in any topological 
space), while the inclusion~$\subseteq$ follows from the fact that 
$X\!\times X$ consists of points isolated in $\scc X\!\times\scc X$ 
(as easy to see, $\cl(A\setminus B)\subseteq\cl(\cl A\setminus\cl B)$ 
whenever $A$~consists of isolated points).
\end{proof}


\subsection*{Left and right closure operations}

To describe topologically the extension~$\wt{\;}\;$, we provide 
more special topological constructions. If $R\subseteq X\!\times Y$, 
let $R^{(x)}$ and $R_{(y)}$ denote the {\it left\/} and {\it right\/} 
{\it sections\/} of~$R$ given by $x\in X$ and $y\in Y$ resp.:
$$
R^{(x)}=R\cap(\{x\}\times Y) 
\quad\text{ and }\quad
R_{(y)}=R\cap(X\times\{y\}).
$$
As easy to see, $\bigl((R^{-1})^{(y)}\bigr)^{-1}=R_{(y)}$ 
and $\bigl((R^{-1})_{(x)}\bigr)^{-1}=R^{(x)}$.

Given a~topology on the set $X\times Y$, we introduce two closure-like 
operations on its subsets, the {\it left closure\/} and the {\it right 
closure\/}, denoted by $\lcl$ and~$\rcl$ resp.:
$$
\lcl R=
\bigcup_{x\in X}\cl R^{(x)}
\quad\text{ and }\quad
\rcl R=
\bigcup_{y\in Y}\cl R_{(y)}.
$$
Note that in the first union we could take only $x\in\dom R$, and 
in the second, only $y\in\ran R$. Observe also that if all the sets 
$\{x\}\times Y$ are closed in $X\times Y$ (e.g.~if the topology on 
$X\times Y$ is the standard product of a~$T_1$-topology on $X$ and 
an arbitrary topology on~$Y$), then we could replace $\cl R^{(x)}$ 
by $\cl_{\{x\}\times Y}R^{(x)}$; and similarly for the right closure.


\begin{lemma}\label{3.2}
Let $X\!\times Y$ be a~topological space 
and $R,S\subseteq X\!\times Y$.
\\
\rule{0em}{1.5em}%
1. 
$\lcl\emptyset=\emptyset$,\, 
$R\subseteq\lcl R$,\, 
$\lcl(\lcl R)=\lcl R$, and\, 
$\lcl R\cup\lcl S=\lcl(R\cup S)$.
\\
\rule{0em}{1em}%
2. 
$\rcl\emptyset=\emptyset$,\, 
$R\subseteq\rcl R$,\, 
$\rcl(\rcl R)=\rcl R$, and\, 
$\rcl R\cup\rcl S=\rcl(R\cup S)$.
\\
\rule{0em}{1em}%
3. 
$(\lcl R^{-1})^{-1}=\rcl R$\, and\, 
$(\rcl R^{-1})^{-1}=\lcl R$.
\end{lemma}

\begin{proof}
1. 
The three first statements are obvious, for the fourth 
we argue as follows:
\begin{align*}
\lcl R\cup\lcl S
&=\Bigl(\,\bigcup_{x\in X}\cl R^{(x)}\Bigr)
\,\cup\,\Bigl(\,\bigcup_{x'\in X}\cl S^{(x')}\Bigr)\,
\\
&=\bigcup_{x,x'\in X}\bigl(\cl R^{(x)}\cup\cl S^{(x')}\bigr)
=\bigcup_{x\in X}\bigl(\cl R^{(x)}\cup\cl S^{(x)}\bigr)
\\
&=\bigcup_{x\in X}\cl\bigl(R^{(x)}\cup S^{(x)}\bigr)
=\bigcup_{x\in X}\cl(R\cup S)^{(x)}
=\lcl(R\cup S).
\end{align*}

2. 
Dually.

3. 
As ${}^{-1}$~distributes w.r.t.~arbitrary unions, we have:
\begin{align*}
(\lcl R^{-1})^{-1}
&=\Bigl(\,\bigcup_{y\in Y}\cl\bigl((R^{-1})^{(y)}\bigr)\Bigr)^{-1}
=\bigcup_{y\in Y}\bigl(\cl\bigl((R^{-1})^{(y)}\bigr)\bigr)^{-1}
\\
&=\bigcup_{y\in Y}\cl\bigl((R^{-1})^{(y)}\bigr)^{-1}
=\bigcup_{y\in Y}\cl R_{(y)}
=\rcl R.
\end{align*}
Here ${}^{-1}$ and~$\cl{}$ commute in the sense that the implied 
topology on $Y\!\times X$ consists of sets~$S$ such that $S^{-1}$ 
belong to the considered topology on $X\!\times Y$.
The second formula is proved dually.
\end{proof}


Clause~1 of Lemma~\ref{3.2} states that the left closure is indeed a~closure 
operation. Given a~topology~$\tau$ on $X\!\times Y$, we let $\tau_\lcl$ 
to denote the topology defined by~$\lcl$, i.e.~with closed sets of form 
$S=\lcl S$. Clearly, $\tau_\lcl$~refines~$\tau$ and, in general, is stronger. 
Moreover, if all sets $\{x\}\!\times Y$ are $\tau$-closed (as happens e.g.~if
$\tau$~is the standard product of a~$T_1$-topology on~$X$ and a~topology on~$Y$), 
then they are $\tau_\lcl$-clopen, so $X\!\times Y$ endowed with~$\tau_\lcl$ 
is their topological sum.

Clause~2 states that the right closure also defines a~topology~$\tau_\rcl$ 
with the dual properties. The topology generated by $\tau_\lcl\cup\tau_\rcl$ 
is discrete in natural cases (e.g.~if $\tau$~is the product of $T_1$-topologies 
on $X$ and~$Y$) while $\tau_\lcl\cap\tau_\rcl$ refines~$\tau$ and, in general, 
still is stronger (e.g.~if $\mathbb R\!\times\mathbb R$ is the real plane with 
its usual topology and $S$~is $=_\mathbb R\!\setminus\{(0,0)\}$, then 
$S=\lcl S=\rcl S\ne\cl S$).

We shall say that $R\subseteq X\!\times Y$ is {\it left closed on\/} 
$A\subseteq X$ iff the left sections~$R^{(x)}$ are closed for all $x\in A$, 
and {\it right closed on\/} $B\subseteq Y$ iff the right sections~$R_{(x)}$ 
are closed for all $y\in B$. The terms {\it left open on\/} a~set, {\it right 
clopen on\/} a~set, etc.~have the expected meaning.
E.g.~$R\subseteq X\!\times Y$ is right open on~$Y$ iff $R\in\tau_\rcl$.
(In fact, the right versions of these concepts were defined 
in~\cite{Saveliev} for arbitrary $n$-ary relations.)


\begin{proposition}\label{prop lcl rcl}
In general, $\rcl\circ\lcl$ and $\lcl\circ\rcl$ are not closure operators. 
Moreover, there exists $R\subseteq\omega\times\omega$ such that
$$
\lcl(\rcl(\lcl R))\ne\rcl(\lcl R)
\quad\text{ and }\quad
\rcl(\lcl(\rcl R))\ne\lcl(\rcl R)
$$
in the space $\scc\omega\times\scc\omega$ with the product topology.
\end{proposition}

\begin{proof}
Let us first show that the usual order~$\ge$ on~$\omega$ satisfies the first of 
the two inequalities: $\lcl(\rcl(\lcl({\ge})))\ne\rcl(\lcl({\ge}))$. We have:
\begin{align*}
\lcl({\ge})=
\bigcup_{u\in\scc\omega}\cl\bigl({\ge}^{(u)}\bigr)=
\bigcup_{u\in\scc\omega}\cl\bigl({\ge}\,\cap(\!\{u\}\!\times\scc\omega)\bigr)=
\bigcup_{m\in\omega}\cl\bigl({\ge}\,\cap(\!\{m\}\!\times\omega)\bigr).
\end{align*}
Since for any fixed $m\in\omega$ the set ${\ge}\,\cap\{m\}\!\times\omega=
\{(m,n):m\ge n\}$ is finite, it is closed. Hence the latter union is equal to
$\bigcup_{m\in\omega}\{(m,n):m\ge n\}={\ge}$. Thus we have
$\lcl({\ge})={\ge}$ and so $\rcl(\lcl({\ge}))=\rcl({\ge})$.
Next, 
\begin{align*}
\rcl({\ge})
&=
\bigcup_{v\in\scc\omega}\cl\bigl({\ge}_{(v)}\bigr)=
\bigcup_{v\in\scc\omega}\cl\bigl({\ge}\,\cap(\scc\omega\!\times\!\{u\})\bigr)=
\bigcup_{m\in\omega}\cl\bigl({\ge}\,\cap(\omega\!\times\!\{m\})\bigr)
\\
&=
\bigcup_{m\in\omega}\cl\bigl((\omega\setminus m)\!\times\!\{m\}\bigr)=
\bigcup_{m\in\omega}\scc(\omega\setminus m)\!\times\!\{m\}.
\end{align*}
Finally,
\begin{align*}
\lcl(\rcl({\ge}))
&=
\bigcup_{u\in\scc\omega}\cl\bigl(\rcl({\ge})^{(u)}\bigr)=
\bigcup_{u\in\scc\omega}\cl\bigl(\rcl({\ge})\cap(\!\{u\}\!\times\scc\omega)\bigr)
\\
&=
\bigcup_{u\in\scc\omega}\cl\Bigl(
\Bigl(\bigcup_{m\in\omega}\scc(\omega\setminus m)\!\times\!\{m\}\Bigr)
\cap(\!\{u\}\!\times\scc\omega)\Bigr).
\end{align*}
Note that for any fixed $u\in\scc\omega$,
$$
\Bigl(\bigcup_{m\in\omega}\scc(\omega\setminus m)\!\times\!\{m\}\Bigr)
\cap(\!\{u\}\!\times\scc\omega)=
\{(u,m):m\in\omega\wedge u\notin m\},
$$
and this set equals $\{(n,m):m\le n\}$ if $u$~is principal, and 
$\{(u,m):m\in\omega\}=\{u\}\!\times\omega$ otherwise. Hence, the closure
of this set equals $\{(n,m):m\le n\}$ if $u$~is principal, and 
$\{u\}\!\times\scc\omega$ otherwise. Therefore, the union of these sets
for all $u\in\scc\omega$ is
\begin{align*}
\bigcup_{u\in\scc\omega}\cl\{(u,m):m\in\omega\wedge u\notin m\}=
{\ge}\,\cup((\scc\omega\setminus\omega)\!\times\scc\omega).
\end{align*}
The latter set properly includes $\rcl(\ge)$, thus showing
$\lcl(\rcl(\ge))\ne\rcl(\ge)$, as reqired.

Furthermore, a~dual argument shows that $\rcl(\le)={\le}$ and 
$\rcl(\lcl(\le))\ne\lcl(\le)$. Now, let $A$ and~$B$ be two disjoint 
copies of the discrete space~$\omega$, e.g.~consisting of even and odd 
natural numbers resp., and let $R$~be the union of ${\ge}$ on~$A$ and ${\le}$ 
on~$B$. Then $R$~satisfies both required inequalities, completing the proof.
\end{proof}

This observation can be improved in two ways. First, $\omega$~can be replaced 
with any infinite discrete~$X$. Second, it can be shown that for many 
ordinals~$\alpha$, even $\alpha$th iterations of $\rcl\circ\lcl$ and 
$\lcl\circ\rcl$ are not closure operators. We leave this for the reader.


\subsection*{Characterization of $\wt R$ via left and right closures}

Now we provide a~topological characterization of the extension~$\,\wt{}\;\,$ 
by using the same product topology on $\scc X\times\scc X$ as for the 
extension~${}^*\,$ but in terms of the left and right closures.

\begin{theorem}\label{thm lcl rcl}
If $R\subseteq X\!\times X$ then
$\wt R=\rcl(\lcl R)$ and $(\wt{R^{-1}})^{-1}=\lcl(\rcl R)$.
\end{theorem}

\begin{proof}
The first formula was in fact proved in~\cite{Saveliev} (see Section~3 there) 
though without using of the terms ``left and right closures".
The second formula follows from the first one:
$$
(\wt{R^{-1}})^{-1}=
(\rcl(\lcl R^{-1}))^{-1}=
\lcl((\lcl R^{-1})^{-1})=
\lcl(\rcl R).
$$
Two last equalities here twice use clause~3 of Lemma~\ref{3.2}.
\end{proof}

Clearly, the $\wt{\;\;}\,$-extension can be expressed via 
only the left closure (of only the right closure) combined 
with the inversion:

\begin{corollary}\label{coro lcl rcl}
If $R\subseteq X\!\times X$ then
\begin{align*}
\wt R&=
(\lcl(\lcl R)^{-1})^{-1}= 
\rcl((\rcl R^{-1})^{-1}),
\\
(\wt{R^{-1}})^{-1}&=
\lcl((\lcl R^{-1})^{-1})=
(\rcl(\rcl R)^{-1})^{-1}.
\end{align*}
\end{corollary}

\begin{proof}
Theorem~\ref{thm lcl rcl} and Lemma~\ref{3.2}.
\end{proof}


More topological properties of both types of extensions:

\begin{proposition}
Let $R\subseteq X\!\times X$.
\\
\rule{0em}{1.5em}%
1.
$R^*$~is left and right clopen on~$\scc X$; in general, 
it is not open in $\scc X\!\times\scc X$.
\\
\rule{0em}{1em}%
2.
$-((-R)^*)$~is left and right clopen on~$\scc X$; in general, 
it is not closed in $\scc X\!\times\scc X$.
\\
\rule{0em}{1em}%
3.
$\wt R$~is right clopen on~$\scc X$; it is also left clopen on~$X$ 
but, in general, neither left closed nor left open on~$\scc X$.
\\
\rule{0em}{1em}%
4.
$(\wt{R^{-1}})^{-1}$~is left clopen on~$\scc X$; it is also right clopen 
on~$X$ but, in general, neither left closed nor left open on~$\scc X$.
\end{proposition}

\begin{proof}
Let us comment only that if $R$~is $=_X$ then $R^*$~is not open, 
while if $R$~is $(X\!\times X)\setminus\!=_X$ then $-((-R)^*)$ is not 
closed, and that $\wt<$ and~$(\wt>)^{-1}$ are neither left closed nor 
left open on~$\scc X$ (the examples in the proof of Theorem~\ref{2.2}).
\end{proof}


\subsection*{Characterizations of $\wt R$ and $R^*$ 
via continuous extensions of homomorphisms}

\begin{theorem}\label{?}
Let $X$~be a~discrete space and $R\subseteq X\times X$, and let $Y$~be 
a~compact Hausdorff space and $S\subseteq Y\times Y$ right closed in the 
product topology on $Y\times Y$. Then the continuous extension~$\wt h$ 
of any homomorphism~$h$ of $(X,R)$ into $(Y,S)$ is a~homomorphism of 
$(\scc X,\wt R)$ into $(Y,S)$. Moreover, $\wt R$~is a~unique extension 
of~$R$ with this property.
\end{theorem}

\begin{proof}
Theorem~4.2 in~\cite{Saveliev}.
\end{proof}

\begin{theorem}\label{??}
Let $X$~be a~discrete space and $R\subseteq X\times X$, and let 
$Y$~be a~compact Hausdorff space and $S\subseteq Y\times Y$ closed in 
the product topology on $Y\times Y$. Then the continuous extension~$\wt h$ 
of any homomorphism~$h$ of $(X,R)$ into $(Y,S)$ is a~homomorphism of 
$(\scc X,R^*)$ into $(Y,S)$. Moreover, $R^*$~is a~unique extension
of~$R$ with this property.
\end{theorem}

\begin{proof}
Fix any $u,v\in\scc X$ and show that $R^*(u,v)$ implies
$S(\wt h(u),\wt h(v))$. Since $S$~is closed in the product topology 
on $Y\times Y$, it suffices to verify that for any neighborhood~$U$ 
of $\wt h(u)$ and any neighborhood~$V$ of $\wt h(v)$, there exist
$y'\in U$ and $y''\in V$ with $S(y',y'')$.

By continuity of $\wt h:\scc X\to Y$, $\wt h^{-1}U$ is a~neighborhood 
of~$u$ and $\wt h^{-1}V$ is a~neighborhood of~$v$. Fix any $A\in u$ 
such that $\wt A\subseteq\wt h^{-1}U$ (where $\wt A=\cl_{\scc X}A$). 
By $R^*(u,v)$, we get $RA\in v$, thus $\wt{RA}$ is a~neighborhood of~$v$, 
and hence, so is 
$$
\wt{RA}\cap\wt h^{-1}V.
$$ 
Therefore, there are $x'\in A$ and $x''\in h^{-1}V$ such that $R(x',x'')$.
So we have $h(x')\in hA\subseteq U$ and $h(x'')\in V$, and as $h$~is
a~homomorphism, $S(h(x'),h(x''))$. Thus letting $y'=h(x')$ and $y''=h(x'')$, 
we get $y'\in U$, $y''\in V$, and $S(y',y'')$, as required.

That $R^*$~is a~unique extension of~$R$ having this property easily follows 
from the uniqueness of the Stone--\v{C}ech compactification of~$X$ and the 
fact that $R^*$~is closed in the product topology on $\scc X\times\scc X$ 
(Theorem~\ref{3.1}).
\end{proof}


\subsection*{Generalizations to topological spaces}

The topological characterizations of $R^*$ and $\wt R$ allow 
to expand the ${}^*$ and $\,\wt{\;}\,$~extensions from relations~$R$ 
on~$X$ to relations~$R$ on~$\scc X$, and further, on any topological 
spaces. Thus instead of the ${}^*$ and $\,\wt{\;}\,$ operators 
we can consider the $\cl$ and $\rcl\circ\lcl$ operators
in the product topology on the square of the given space.

\begin{theorem}\label{2.3}
Let\, $X$~be a~topological space and $R,S\subseteq X\!\times X$.
The following equalities hold:
\begin{align*}
\cl(R\cup S)&=\cl R\cup\cl S,
&
\rcl(\lcl(R\cup S))&=\rcl(\lcl R)\cup\rcl(\lcl S),
\\
\cl(R^{-1})&=(\cl R)^{-1},
&
\cl(R\circ S)&=\cl R\circ\cl S\text{ if $X$~is compact }T_2.
\end{align*}
Moreover, the following inclusions hold:
\begin{align*}
-(\cl R)&\subseteq\cl(-R),
&
-(\rcl(\lcl R))&\subseteq\rcl(\lcl(-R)),
\\
\cl(R\cap S)&\subseteq\cl R\cap\cl S,
&
\rcl(\lcl(R\cap S))&\subseteq\rcl(\lcl R)\cap\rcl(\lcl S),
\\
&&\cl(R\circ S)\subseteq\cl R\circ\cl S,&
\end{align*}
neither of them, in general, cannot be replaced by an equality. 
Finally, $\rcl(\lcl(R^{-1})$ and $(\rcl(\lcl R))^{-1}$, in general, 
are $\subseteq$-incomparable.
\end{theorem}

\begin{proof}
The claims about interplay with Boolean operations are trivial.

Show that $\cl(R^{-1})=(\cl R)^{-1}$ holds for any topological 
space~$X$. Indeed, $(a,b)\in\cl(R^{-1})$ means that for all 
neighborhoods~$O$ of $(a,b)$ there is $(a',b')\in R^{-1}\cap O$. 
Since any neighborhood in the product topology includes a~basic 
neighborhood of form $U\!\times V$ for some $U$ and~$V$ open 
in~$X$, the latter means that for all neighborhoods $U$ of~$a$ 
and $V$ of~$b$ there are $a'$ and~$b'$ such that $a'\in U$, 
$b'\in V$, and $(a',b')\in R^{-1}$, i.e. $(b',a')\in R$. 
But this is clearly equivalent to $(b,a)\in\cl R$, 
and thus to $(a,b)\in(\cl R)^{-1}$, as required.

Show that the inclusion $\cl(R\circ S)\subseteq\cl R\circ\cl S$ 
holds in any topological space. Indeed, on the one hand, the formula
$(a,b)\in\cl(R\circ S)$ is clearly equivalent to the assertion that 
for all neighborhoods $U$ of~$a$ and $V$ of~$b$ there are $a',b',c$ 
such that $a'\in U$, $b'\in V$, $S(a',c)$, and $(c,b')\in R$.

On the other, the formula $(a,b)\in\cl R\circ\cl S$ means that 
there is~$c$ with $(a,c)\in\cl S\wedge(c,b)\in\cl R$. The first 
conjunct is equivalent to the assertion that for all neighborhoods 
$U$ of~$a$ and $W'$ of~$c$ there are $a',c'$ such that $a'\in U$, 
$c'\in W_1$, and $S(a',c')$; and the second, that for all 
neighborhoods $W''$ of~$c$ and $V$ of~$b$ there are $c'',b'$ 
such that $c''\in W''$, $b'\in V$, and $R(c'',b')$. Letting 
$W=W'\cap W''$, we see that the formula is equivalent to the 
existence of~$c$ such that for all neighborhoods $U$ of~$a$, 
$W$ of~$c$, $V$ of~$b$ there are $a',c',c'',b'$ satisfying
$a'\in U$, $\{c',c''\}\subseteq W$, $b'\in V$, and 
$S(a',c')\wedge R(c'',b')$. This is clearly weaker that 
the first formula, which shows the required inclusion.


Show now that if $X$~is compact Hausdorff then the converse inclusion 
$\cl R\circ\cl S\subseteq\cl(R\circ S)$ holds as well, thus getting 
the equality $\cl(R\circ S)=\cl R\circ\cl S$. 
As easy to see, this is equivalent to the following assertion:
\begin{quote}
If $X$~is compact Hausdorff and both $R$ and $S$ are closed in the space 
$X\times X$ with the standard product topology, then so is $R\circ S$.
\end{quote}
So let us verify that, for any closed $R$ and~$S$, we have $(a,b)\in R\circ S$ 
whenever $(a,b)\in\cl(R\circ S)$, i.e.~whenever for all neighborhoods $U$ of~$a$ 
and $V$ of~$b$ there are $a',b',c$ such that $a'\in U$, $b'\in V$, $(a',c)\in S$ 
and $(c,b')\in R$.

For any open $U\ni a$ let 
$$
A_U=
\{x\in X:(\exists a'\in\cl U)\,(a',x)\in S\}.
$$
Thus $A_U=\pr_1((\cl U\times X)\cap S)$, 
and so $A_U$~is a~non-empty closed set.
(It is closed as a~continuous image of 
the compact set $(\cl U\times X)\cap S$ 
into the Hausdorff space~$X$.)
Similarly, for any open $V\ni b$ let 
$$
B_V=
\{x\in X:(\exists b'\in\cl V)\,(x,b')\in R\}.
$$
Thus $B_V=\pr_2((X\times \cl V)\cap R)$, 
and so $B_V$~is a~non-empty closed set.
Now let $C_{U,V}=A_U\cap B_V$. Thus 
\begin{align*}
C_{U,V}
&=
\{x\in X:
(\exists a'\in\cl U)\,(a',x)\in S
\,\wedge\, 
(\exists b'\in\cl V)\,(x,b')\in R
\}
\\
&=
\{x\in X:
(\exists a'\in\cl U)\,
(\exists b'\in\cl V)\,
(a',x)\in S\,\wedge\,(x,b')\in R
\},
\end{align*}
and so $C_{U,V}$ is a~non-empty closed set, too.

Let $\mathcal C=\{C_{U,V}:a\in U,b\in V,\text{ and $U,V$ are open}\}.$
The family~$\mathcal C$ is centered (i.e.~has the finite intersection 
property) since, for any $n<\omega$,
$$
\bigcap_{i<n}C_{U_i,V_i}=C_{U,V}\ne\emptyset
\quad\text{ where }\;
U=\bigcap_{i<n}U_i,\;V=\bigcap_{i<n}V_i.
$$
Therefore, since $X\times X$ is compact, 
$\bigcap\mathcal C\ne\emptyset.$
Pick any $c\in\bigcap\mathcal C$. We have got: 
for all neighborhoods $U$ of~$a$ 
and $V$ of~$b$ there are $a'\in\cl U$ and $b'\in\cl V$ such that 
$(a',c)\in S$ and $(c,b')\in R$. Moreover, a~stronger conclusion 
is true: there are $a'\in U$ and $b'\in V$ such that $(a',c)\in S$ 
and $(c,b')\in R$. Indeed, since $X$~is compact Hausdorff, it is 
regular, and so for any neighborhood~$O$ of any $x\in X$ 
there is a~neighborhood~$O'$ of~$x$ with $\cl O'\subseteq O$; 
hence we can replace $\cl U$ and $\cl V$ with $U$ and~$V$ resp.

Now note that, since $S$~is closed and 
for all neighborhoods $U$ of~$a$ there is 
$a'\in U$ with $(a',c)\in S$, we have $(a,c)\in S$. 
Similarly using that $R$~is closed, we get $(c,b)\in R$. 
Together we obtain $(a,c)\in S\wedge(c,b)\in R$, 
and thus $(a,b)\in R\circ S$, as required.

Finally, let us show that we can weaken neither compactness 
to local compactness, nor $T_2$- to $T_1$-axiom, even for $R=S$. 
Indeed, if $R$~is the hyperbole 
$\{(x,x^{-1}):x\in\mathbb R\setminus\{0\}\}$, then $R$~is closed 
in the usual (locally compact Hausdorff) topology on the real plane, 
however, $R\circ R$ is ${=_\mathbb R}\setminus\{(0,0)\}$ and so 
is not closed. Also, if we consider the Zariski topology on the 
real plane (whose closed sets are generated by real polynomial curves), 
which is a~compact $T_1$-topology, the same~$R$ provides 
a~second required counter-example.
\end{proof}


\section{Extending relations as multi-valued maps\/}

In this section, we show that ${}^*$-extensions of relations
are in fact continuous extensions of appropriate maps.

\subsection*{Continuous extensions of maps\/}

Recall that the ultrafilter extensions of functions, as defined 
in~\cite{Saveliev}, in the case of a~unary function $F:X\to X$ gives 
the standard continuous extension $\wt F:\scc X\to\scc X$ defined by 
letting for all $A\subseteq X$ and $u\in\scc X$,
$$
A\in\wt F(u)
\;\;\lra\;\; 
F^{-1}A\in u.
$$

An easy observation is that, if a~relation~$R$ on~$X$ is functional,
then $R^*$~coincides with its ultrafilter extension as a~map of $X$
into itself:

\begin{lemma}\label{ext map} 
For any function $F:X\to X$, 
$$\wt F=F^*.
$$
\end{lemma}

\begin{proof} 
Clear from the definitions.
\end{proof}

This observation leads to the query whether a~similar fact holds for 
arbitrary (not neccessarily functional) relations. A~natural step in 
achieving this guess is to consider relations $R\subseteq X\times Y$ 
as multi-valued maps of $X$ into~$Y$, which in turn can be understood 
as (one-valued) maps of $X$ into the powerset of~$Y$. Again avoiding 
any new notation, we continue to denote them by the same symbol~$R$;
thus we regard $R$ as the map
$$
R:X\to P(Y)
$$
defined by letting $R(x)=\{y:R(x,y)\}$, for all $x\in X$.


In general, given a~map~$F$ of a~discrete set~$X$ to a~compact 
Hausdorff space~$Y$, its (unique) continuous extension 
$\wt F:\scc X\to Y$ is defined by letting for all $u\in\scc X$,
\begin{align*}
\wt F(u)=y
\text{ such that }
\{y\}=
\bigcap_{A\in u}\cl_{\scc X}FA.
\end{align*}
We point out that the map~$\wt F$ is also closed, i.e.~takes closed sets 
to closed sets (since continuous maps from compact spaces into Hausdorff
spaces always are closed).
If $Y=\scc X$, this coincides with the extension defined above
as it should do (see Lemma~3.2 in~\cite{Saveliev}).

The following easy observation shows that $R^*$ regarded as the map
$$
R^*:\scc X\to P(\scc X)
$$ 
can be expressed by a~similar formula.

\begin{lemma}\label{R*u} 
For any $R\subseteq X\times X$,
\begin{align*}
R^*(u)=\bigcap_{A\in u}\cl_{\scc X}RA.
\end{align*}
\end{lemma}

\begin{proof} 
We have
\begin{align*}
R^*(u)
&=
\{v\in\scc X:R^*(u,v)\}
\\&=
\bigl\{v\in\scc X:(\forall A\in u)\,
\{y:(\exists x\in A)\,R(x,y)\}\in v\bigr\}
\\&=
\bigl\{v\in\scc X:(\forall A\in u)\,
RA\in v\bigr\}
\\&=
\bigcap_{A\in u}\{v\in\scc X:RA\in v\}
=
\bigcap_{A\in u}\wt{RA}
=
\bigcap_{A\in u}\cl_{\scc X}RA.
\end{align*}
Here $\wt{RA}$ is the basic open set of~$\scc X$ given by $RA\subseteq X$.
\end{proof}

This tells in favour of our guess that $R^*$~is a~continuous extension 
of an appropriate map generated by~$R$ though does not tell yet 
what such a~map should be. One can try to start from $R$~itself, 
regarded as the map between the discrete spaces $X$ and~$P(X)$. 
However its continuous extension 
$$
\wt R:\scc X\to\scc P(X)
$$ 
has the range~$\scc P(X)$, while $R^*$ regarded as the map has 
the range~$P(\scc X)$, or rather, its subset consisting of 
closed sets of~$\scc X$. As well-known, the latter set, with any 
reasonable topology on it, is not homeomorphic to $\scc P(X)$.

Before we shall be able to mark the right way, we should make 
some preparations concerning reasonable topologies on the set 
of closed sets of~$\scc X$.


\subsection*{Vietoris topologies 
on the set of non-empty closed sets}

If $X$~is a~topological space, let $P_\cl(X)$~denote
the closed topology minus the empty set:
$$
P_\cl(X)=
\{C\subseteq X:C\text{ is non-empty and closed in }X\}.
$$

For any $A\subseteq X$ let
\begin{align*}
A^-&=
\{B\in P_\cl(X):B\subseteq A\},
\\
A^+&=
\{B\in P_\cl(X):B\cap A\ne\emptyset\}.
\end{align*}
Note that $A^-=P_\cl A$ and that the operations ${}^-$ and~${}^+$ 
are dual in the following: for any $A\subseteq X$, we have
$A^-=P_\cl(X)\setminus(X\setminus A)^+$ and
$A^+=P_\cl(X)\setminus(X\setminus A)^-$.

The family $\{C^{-}:C\in P_\cl(X)\}$ is a~closed subbase of 
the {\it lower Vietoris\/} topology on~$P_\cl(X)$, and 
the family $\{C^{+}:C\in P_\cl(X)\}$, of 
the {\it upper Vietoris\/} topology on the same set.
In fact, the latter family is a~closed base since it is closed 
under finite unions. Indeed, for any family $\{C_i:i\in I\}$ 
of subsets of~$X$ we clearly have 
$
\bigcup_iC_{i}^+=
\bigl(\bigcup_iC_i\bigr){}^+,
$
and if the~$C_i$ are closed and the family is finite, 
then the latter set is of form~$C^+$ for a~closed set~$C$.
For $\tau$~an (open) topology of~$X$, we shall denote 
the lower and upper Vietoris topologies on~$P_\cl(X)$ 
by $\tau^{-}$ and~$\tau^{+}$ resp.

Observe that $\{O^+:O\in\tau\}$ is an open base of~$\tau^-$,
and $\{O^-:O\in\tau\}$ is an open subbase of~$\tau^+$, and that
for all $S\subseteq P_\cl(X)$, we have
\begin{align*}
\cl_{\tau^-}S
&=
\bigcap\Bigl\{
\bigcup_{C\in F}C^-:
F\in P_\omega(P_\cl(X))
\wedge 
S\subseteq\bigcup_{C\in F}C^-
\Bigr\},
\\
\cl_{\tau^+}S
&=
\bigcap\bigl\{
C^+:C\in P_\cl(X)\wedge S\subseteq C^+
\bigr\},
\end{align*}
where, as usual, $P_\omega(Y)$ is the set of finite 
subsets of~$Y$.

The {\it Vietoris\/} topology on $P_\cl(X)$ is the supremum of 
these two. It can be given by the open base consisting of sets
\begin{align*}
\langle F\rangle
=
\Bigl\{
B\in P_\cl(X):
B\subseteq\bigcup F
\wedge
(\forall O\in F)\,B\cap O\ne\emptyset
\Bigr\}
=
\bigcap_{O\in F}O^-
\cap
\bigcap_{O\in F}O^+,
\end{align*}
for all $F\in P_\omega(\tau_X)$ where $\tau_X$~is the original
(open) topology on~$X$. Letting $\tau$~to denote the Vietoris topology, 
we clearly have for all $S\subseteq P_\cl(X)$, 
\begin{align*}
\cl_\tau\,S\,=\,
\cl_{\tau^-}S\,\cap\,\cl_{\tau^+}S
\end{align*}
(by a~general fact that the closure in the supremum of a~family 
of topologies is the intersection of the closures in each of the 
topologies in the family).

As well-known, if $X$~is compact Hausdorff then so is
$P_\cl(X)$ with the Vietoris topology. In particular, 
for any discrete~$X$ the space $P_\cl(\scc X)$ is compact 
Hausdorff. It is however not extremally disconnected 
(as e.g.~it contains countable convergent sequences).
Also if $X$~is $T_1$-space, then $P_\omega(X)$ is dense
in $P_\cl(X)$. In particular, $P_\omega(\scc X)$ is dense
in $P_\cl(\scc X)$, whence it easily follows that even
$P_\omega(X)$ is dense in $P_\cl(\scc X)$.

A~number of several facts about the Vietoris topology can 
be found in~\cite{Michael}. For a~more general construction 
of Vietoris-type topologies, see~\cite{Poppe}.


\subsection*{Vietoris topologies on the set of filters}

If $X$~is a~set, let $\filt X$ denote the set of filters
over~$X$, and if $D\in\filt X$, let $\ol D$~denote the set of 
ultrafilters extending it: $\ol D=\{u\in\scc X:D\subseteq u\}$.

\begin{lemma}\label{filters vs closed sets} 
Let $X$~be arbitrary set.
\\
\rule{0em}{1.5em}%
1. 
If $D\in\filt X$, then $\ol D\in P_\cl(\scc X)$.
\\
2. 
If $S\subseteq\scc X$, then $\bigcap S\in\filt X$ 
and moreover $\ol{\bigcap S}=\cl_{\scc X}S$.
\\
\rule{0em}{1.5em}%
Consequently, $\bigcap$~is an anti-isomorphism 
of $(P_\cl(\scc X),\subseteq)$ onto $(\filt X,\subseteq)$, and 
${}\ol{\phantom A}$~is the inverse of it.
\end{lemma}

\begin{proof} 
See e.g.~\cite{Hindman Strauss}, Theorem~3.20.
\end{proof}

Thus $P_\cl(\scc X)$ can be identified with~$\filt X$. This allows 
to endow~$\filt X$ with the lower, upper, and standard Vietoris topologies.
As the inclusion of closed subsets of~$\scc X$ corresponds to the converse 
inclusion of filters over~$X$, we redefine: For any $C\in\filt X$, let now 
$C^{-},C^{+}\subseteq\filt X$ be the following sets:
\begin{align*}
C^-&=
\{B\in\filt X:C\subseteq B\},
\\
C^+&=
\{B\in\filt X:B\cup C\text{ is centered}\}
\end{align*}
(notice that $\ol B\subseteq\ol C$ iff $C\subseteq B$, and 
$\ol B\cap\ol C\ne\emptyset$ iff $B\cup C$~is centered).

The following lemma describes the closure in the Vietoris 
topologies on~$\filt X$.


\begin{lemma}\label{Vietoris closures} 
Let $S$~be a~set of filters over~$X$. 
\\
\rule{0em}{1.5em}%
1. 
$\cl_{\tau^-}S$~consists of all filters~$D$ such that
for every finite set $F\subseteq\filt X$, 
if each filter in~$S$ extends some $C\in F$
then $D$~extends some $C\in F$.
\\
2.
$\cl_{\tau^+}S$~consists of all filters~$D$ such that
for every filter~$C$, 
if each filter in~$S$~is compatible with~$C$
then $C$ and~$D$ are compatible.
\\ 
3. 
$\cl_{\tau}\,S$~consists of all filters~$D$ such that
for every finite set $F\subseteq\filt X$, 
if each filter in~$S$ extends some $C\in F$
then $D$~extends some $C\in F$, and 
if each filter in~$S$~is compatible with some $C\in F$
then $C$ and~$D$ are compatible.
\end{lemma}

\begin{proof} 
1.
A~filter~$D$ belongs to~$\cl_{\tau^-}S$ iff it belongs
to any closed basic set that includes~$S$. 
The latter means that for all $F\in P_\omega(\filt X)$, 
$S\subseteq\bigcup_{C\in F}C^-$ implies
$D\in\bigcup_{C\in F}C^-$, i.e.
$$
S\subseteq\bigcup_{C\in F}\{B\in\filt X:C\subseteq B\}
\;\text{ implies }\;
D\in\bigcup_{C\in F}\{B\in\filt X:C\subseteq B\},
$$
in other words,
$(\forall E\in S)\,(\exists C\in F)\,C\subseteq E$ 
implies $(\exists C\in F)\,C\subseteq D$, as required.

2. 
This is handled likewise and in fact easier:
A~filter~$D$ belongs to~$\cl_{\tau^+}S$ iff for all filters~$C$,
$S\subseteq C^+$ implies $D\in C^+$, i.e.
$(\forall E\in S)\,\ol C\cap\ol E\ne\emptyset$ implies
$\ol C\cap\ol D\ne\emptyset$, as required.

3. 
Since $\cl_\tau\,S\,=\,\cl_{\tau^-}S\,\cap\,\cl_{\tau^+}S$,
this follows from clauses~1,~2.
\end{proof}


\subsection*{Continuity of $R^*$}

Given $R:X\to P(X)$, define $R^{\centerdot}:X\to P_\cl(\scc X)$
as follows:
$$
R^\centerdot=\cl_{\scc X}\circ R.
$$
In other words, for all $x\in X$ we have
$
R^\centerdot(x)=\cl_{\scc X}R(x),
$
thus 
$$
v\in R^\centerdot(x)
\;\;\lra\;\;
v\in\cl_{\scc X}R(x)
\;\;\lra\;\;
v\in\wt{R(x)}
\;\;\lra\;\;
R(x)\in v
$$
for any $v\in\scc X$.
Note that the values of~$R^\centerdot$ are basic clopen sets 
in~$\scc X$ and that $R^\centerdot$~regarded as the relation
$$
R^\centerdot=
\{(x,v):v\in R^\centerdot(x)\}=
\{(x,v):R(x)\in v\}
$$ 
coincides with $\lcl R$, the left closure of~$R$, introduced in Section~3.

Since the Vietoris topology on $P_\cl(\scc X)$ is compact Hausdorff,
$R^\centerdot$~uniquely extends to the continuous map
$\wt{R^\centerdot}:\scc X\to P_\cl(\scc X)$ by the aforementioned rule:
$$
\wt{R^\centerdot}(u)=C
\text{ such that }
\{C\}=
\bigcap_{A\in u}\cl_{P_\cl(\scc X)}R^\centerdot A
$$
for all $u\in\scc X$.

It may be noticed that this formula looks quite similar to the formula 
$$
R^*(u)=
\bigcap_{A\in u}\cl_{\scc X}RA
$$
obtained in Lemma~\ref{R*u} for $R^*$~regarded as the map 
$R^*:\scc X\to P_\cl(\scc X)$. The difference is that now we deal 
with the sets $R^\centerdot A\subseteq P_\cl(\scc X)$ and their closures 
in the Vietoris space $P_\cl(\scc X)$, rather than with the sets 
$RA\subseteq X(\subseteq\scc X)$ and their closures in~$\scc X$.

(To avoid misreading, let us comment that $R^\centerdot A$~here 
denotes the image of~$A$ under the map~$R^\centerdot$, i.e. 
$\{R^\centerdot(x):x\in A\}$, while $RA$~there denoted the image 
of~$A$ under the relation~$R$, i.e. $\{y:(\exists x\in A)\,R(x,y)\}$.
We emphasize that the image of~$A$ under the map~$R^\centerdot$ is 
{\it not\/} the same that the image of~$A$ under the relation 
$R^\centerdot=\{(x,v):v\in R^\centerdot(x)\}$; the latter is 
$\{v:(\exists x\in A)\,R^\centerdot(x,v)\}=
\bigcup_{x\in A}R^\centerdot(x)$.)


The difference of these two formulas, however, is only
the difference of two ways to construct the same object. 
We are going now to show $R^*=\wt{R^\centerdot}\,$, thus 
confirming the guess that $R^*$, regarded as a~map of 
$\scc X$ into $P_\cl(\scc X)$, continuously extends an 
appropriate map generated by~$R$; namely,~$R^\centerdot\,$.

\begin{theorem}\label{continuity}
The maps $R^*$ and $\wt{R^\centerdot}$ coincide.
\end{theorem}

\begin{proof}
Since the continuous extension $\wt{R^\centerdot}$ of~$R^\centerdot$ 
is unique, it suffices to verify that $R^*$ extends~$R^\centerdot$ 
and that it is continuous in the Vietoris topology on $P_\cl(\scc X)$.

The first claim is immediate: For all $x\in X$,
\begin{align*}
R^*(x)
&=\{v\in\scc X:(\forall A\ni x)\,RA\in v\}
\\
&=\{v\in\scc X:R(x)\in v\}
\\
&=\wt{R(x)}
=\cl_{\scc X}R(x)=R^\centerdot(x).
\end{align*}

Before proving that $R^*$~is continuous, recall the following concepts
(due to Kuratowski; see~\cite{Engelking}, 1.7.17).

Let $X$ and $Y$ be topological spaces, and $F:X\to P_\cl(Y)$.
The map~$F$ is {\it lower\/}, resp.~{\it upper $($semi-$)$continuous\/} 
iff for any open $O\subseteq Y$, 
the set $\{x:F(x)\cap O\ne\emptyset\}$, resp.~$\{x:F(x)\subseteq O\}$ 
is open in~$X$. Equivalently, iff for any closed $C\subseteq Y$, 
the set $\{x:F(x)\subseteq C\}$, resp.~$\{x:F(x)\cap C\ne\emptyset\}$ 
is closed in~$X$.

As well-known, $F$~is lower, resp.~upper continuous iff it is 
continuous in the lower, resp.~upper Vietoris topology on $P_\cl(Y)$, 
and $F$~is lower and upper continuous simultaneously iff it is 
continuous in the Vietoris topology on $P_\cl(Y)$ (see~\cite{Michael}).

Therefore, to prove that $R^*$~is continuous, it suffices to show 
that $R^*$~is lower and upper continuous simultaneously.

Show first that $R^*$~is lower continuous. Thus we must check that
for any open $O\subseteq\scc X$, the set $\{u\in\scc X:R^*(x)\cap O
\ne\emptyset\}$ is open in~$\scc X$. W.l.g.~we can pick~$O$ a~basic 
open set: $O=\wt B$ for some $B\subseteq X$. Observe that 
$$
(\exists v)\,R^*(u,v)\cap B\in v
\;\;\lra\;\;
R^{-1}B\in u.
$$
Indeed, the implication from the left to the right holds by 
the definition of~$R^*$. For the converse implication note that
the family $\{RA:A\in u\}$ is centered and pick any $v\in\scc X$ 
extending it; for any $A\in u$ we have $RA\in v$, so $R^*(u,v)$,
and in particular $RR^{-1}B=B\in v$.

Now we get:
\begin{align*}
\bigl\{u\in\scc X:R^*(u)\cap\wt B\ne\emptyset\bigr\}
&=
\bigl\{u\in\scc X:(\exists v)\,R^*(u,v)\cap v\in\wt B\;\bigr\}
\\
&=
\bigl\{u\in\scc X:(\exists v)\,R^*(u,v)\cap B\in v\;\bigr\}
\\
&=
\bigl\{u\in\scc X:R^{-1}B\in u\bigr\}
=
\wt{R^{-1}B},
\end{align*}
thus showing that the set $\{u:R^*(u)\cap\wt B\ne\emptyset\}$ is 
(basic) open in~$\scc X$, which proves the lower continuity of~$R^*$.


Now we must verify that $R^*$~is upper continuous. Before that, 
we introduce the {\it filter extension\/} of relations, used in 
the rest of the  proof but also having an independent interest. 
We shall denote it by the same symbol~${}^*$ since on ultrafilters, 
it gives the previous concept.

For all $R\subseteq X\times X$ and $C,D\in\filt X$, we let
$$
R^*(C,D)
\;\;\lra\;\;
(\forall A\in C)\,
D\cup\{RA\}\text{ is centered}.
$$
(As usual, $S$~is {\it centered\/} iff it has the finite intersection 
property.) Note that for $C,D\in\scc X\subseteq\filt X$, the meaning 
of $R^*(D,C)$ is the same as earlier.

Recall that for $D$~a~filter, $\ol D$~denotes the set of 
ultrafilters extending~$D$. The next result reformulates 
the filter extension in terms of the ultrafilter extension.

\begin{lemma}\label{filterextension}
For all $R\subseteq X\times X$ and $C,D\in\filt X$, 
$$
R^*(C,D)
\;\;\lra\;\;
(\exists u\in\ol C)\,
(\exists v\in\ol D)\,
R^*(u,v).
$$
\end{lemma}

\begin{proof}
The ($\leftarrow$)~part is clear. So assume $R^*(C,D)$, i.e.~that 
$D\cup\{RA\}$~is centered for every $A\in C$, and find $u\in\ol C$ 
and $v\in\ol D$ such that $R^*(u,v)$ in the former meaning, 
i.e.~that $RA\in v$ for each $A\in u$.

Observe that the set $G=D\cup\{RA:A\in C\}$ is centered. Indeed, if 
$m<\omega$ and $A_i\in C$ for all $i<m$, then $A=\bigcap_{i<m}A_i$ 
is in~$C$ and hence $D\cup\{RA\}$~is centered by the assumption. 
Since the inclusion $RA\subseteq\bigcap_{i<m}RA_i$ holds always, 
this shows that $G$~is centered.

Now pick any $v\in\scc X$ extending~$G$. As $D\subseteq G$, we have 
$v\in\ol D$. Consider the family of filters that extend~$C$ and have
the required connection to~$v$:
$$
S=
\{E\in\filt X:C\subseteq E\wedge(\forall A\in E)\,RA\in v\}.
$$
It easy to see that $S$~is closed under $\subseteq$-increasing chains.
By Zorn's lemma, $S$~has a~maximal element~$E$. We verify that $E\in\scc X$.

Assume the converse: $E\notin\scc X$, so there is $A\subseteq X$ 
such that both $E\cup\{A\}$ and $E\cup\{X\setminus A\}$ are centered, 
and hence, can be properly extended to some filters $E'$ and~$E''$, resp. 
Observe that $RX\in v$. Hence, as $R\bigcup_iA_i=\bigcup_iRA_i$ holds always,
$RA\in v$ or $R(X\setminus A)\in v$. If $RA\in v$ then $E'\in S$, and 
if $R(X\setminus A)\in v$ then $E''\in S$; thus in each of both cases,
$E$~is not maximal in~$S$; a~contradiction.

This completes the proof of Lemma~\ref{filterextension}.
\end{proof}

Now we turn back to the proof of the upper continuity of~$R^*$. 
We show that for any closed $C\subseteq\scc X$, the set
$\{u\in\scc X:R^*(u)\cap C\ne\emptyset\}$ is closed in~$\scc X$.
Indeed, let $C\in P_\cl(X)$ and let $E\in\filt X$ be such that 
$C=\ol E$, i.e. $E=\bigcap C$. We have:
\begin{align*}
\bigl\{u\in\scc X:R^*(u)\cap C\ne\emptyset\bigr\}
&=
\bigl\{u\in\scc X:(\exists v\in C)\,R^*(u,v)\bigr\}
\\
&=
\bigl\{u\in\scc X:(\exists v\in C)\,(R^{-1})^*(v,u)\bigr\}
\\
&=
\bigl\{u\in\scc X:(\exists v\in\ol E)\,(R^{-1})^*(v,u)\bigr\}
\\
&=
\bigl\{u\in\scc X:(R^{-1})^*(E,u)\bigr\}
\\
&=
\bigl\{u\in\scc X:(\forall A\in E)\,R^{-1}A\in u\bigr\}
\\
&=
\bigcap_{A\in E}\bigl\{u\in\scc X:R^{-1}A\in u\bigr\}
=
\bigcap_{A\in E}\cl_{\scc X}R^{-1}A.
\end{align*}
Here the second equality uses the identity $(R^{-1})^*=(R^*)^{-1}$
stated in Theorem~\ref{2.2}, and the fourth equality uses 
Lemma~\ref{filterextension}. Therefore, the set 
$\{u:R^*(u)\cap C\ne\emptyset\}$ is closed in~$\scc X$,
which proves the upper continuity of~$R^*$.

The proof of Theorem~\ref{continuity} is now complete.
\end{proof}


\subsubsection*{Acknowledgements.}
I~am grateful to Nikolai L.~Poliakov for useful 
discussions and an assistence in proving some facts.



\end{document}